\documentclass[a4paper]{amsart}
\usepackage{graphicx}
\usepackage{amssymb}
\usepackage{amsmath}
\usepackage{amsthm}
\usepackage{amscd}
\usepackage[all,2cell]{xy}

\usepackage[pagebackref,colorlinks]{hyperref}

\UseAllTwocells \SilentMatrices
\newtheorem{thm}{Theorem}[section]

\newtheorem{cor}[thm]{Corollary}

\newtheorem{lem}[thm]{Lemma}
\newtheorem{exm}[thm]{Example}

\newtheorem{prop}[thm]{Proposition}

\theoremstyle{definition}

\theoremstyle{remark}
\newtheorem{rem}[thm]{\bf Remark}
\numberwithin{equation}{section}

\begin{document}
\title[The singularity category as a stable module category]{The singularity category as a stable module category}
\author[Xiao-Wu Chen, Zhengfang Wang] {Xiao-Wu Chen$^*$, Zhengfang Wang}

\makeatletter
\@namedef{subjclassname@2020}{\textup{2020} Mathematics Subject Classification}
\makeatother

\thanks{$^*$ The corresponding author}
\date{\today}
\subjclass[2020]{16S88, 18G80, 13D02, 16S88}

\thanks{xwchen$\symbol{64}$mail.ustc.edu.cn, zhengfangw$\symbol{64}$nju.edu.cn}
\keywords{singularity category, stable category, stabilization, Leavitt ring, FC ring}

\maketitle

\dedicatory{}%
\commby{}%

\begin{abstract}
We investigate the stabilization $\mathcal{S}$  of the module category over an artinian ring $\Lambda$ by formally inverting the tensor endofunctor  given by the bimodule of relative noncommutative differential $1$-forms. It turns out that $\mathcal{S}$ is a Frobenius abelian category, which is equivalent to the category of finitely presented modules over  the zeroth component $L_0$ of the Leavitt ring $L$. It follows that  $L_0$ is an FC ring in the sense of Damiano, which is usually not quasi-Frobenius.  Moreover, the singularity category of $\Lambda$ is triangle equivalent to the stable module category over $L_0$.
\end{abstract}

\section{Introduction}

Let $\Lambda$ be a left artinian ring, for example, a finite dimensional algebra over a field.  The \emph{singularity category} $\mathbf{D}_{\rm sg}(\Lambda)$ is defined to be the Verdier quotient category of the bounded derived category by the bounded homotopy category of projective modules; see \cite{Buc, Orl}. It detects the homological singularity of $\Lambda$ in the following sense: the singularity category $\mathbf{D}_{\rm sg}(\Lambda)$ vanishes if and only if $\Lambda$ has finite global dimension. We mention that  singularity categories of certain finite dimensional algebras might be equivalent to the one of a commutative algebra of positive Krull dimension \cite{Kalck, KT}.  The latter (graded) singularity category plays a role in  homological mirror symmetry for non-Calabi-Yau cases \cite{Orl, Ebel}.

The singularity category  is called the \emph{stabilized derived category} in \cite{Buc}, as it describes  the stable homological features of $\Lambda$. In other words, it captures the asymptotic  behaviour of the syzygy endofunctor $\Omega_\Lambda$ on the stable module category $\Lambda\mbox{-\underline{\rm mod}}$. This statement is made precise by a fundamental result in \cite{Buc, KV}, which  states that $\mathbf{D}_{\rm sg}(\Lambda)$ is equivalent to the stabilization \cite{Hel68, Tie} of $\Lambda\mbox{-\underline{\rm mod}}$ by formally inverting $\Omega_\Lambda$; see also \cite{Bel}.

Assume that $E\subseteq \Lambda$ is a semisimple subring such that $\Lambda$ is finitely generated as a left $E$-module. We have the $\Lambda$-$\Lambda$-bimodule $\Omega_{{\rm nc}, \Lambda/E}$ of $E$-relative noncommutative differential $1$-forms \cite{CQ}. It is well known that $\Omega_\Lambda$ admits an exact lift, which is given by the tensor endofunctor $(\Omega_{{\rm nc}, \Lambda/E})\otimes_\Lambda - $ on the module category $\Lambda\mbox{-mod}$. It is natural to investigate the category
$$\mathcal{S}=\mathcal{S}(\Lambda\mbox{-mod}, (\Omega_{{\rm nc}, \Lambda/E})\otimes_\Lambda - ),$$
which is the stabilization of $\Lambda\mbox{-mod}$ by formally inverting $(\Omega_{{\rm nc}, \Lambda/E})\otimes_\Lambda - $.

The following result is obtained by combining  the mentioned fundamental result in \cite{Buc, KV, Bel} with Propositions~\ref{prop:Frob} and~\ref{prop:tri-equ}.
\vskip 5pt

\noindent {\bf Proposition~A}.  \emph{The category $\mathcal{S}$ above is Frobenius abelian. Moreover, its stable category $\underline{\mathcal{S}}$ is triangle equivalent to $\mathbf{D}_{\rm sg}(\Lambda)$.
}

\vskip 5pt
Here, we recall from \cite{Hap} that the stable category of any Frobenius exact category is canonically triangulated; compare \cite{Hel}.

Recall from \cite{CWKW} the \emph{Leavitt ring} $L=L_\Lambda(\Omega_{{\rm nc}, \Lambda/E})$ associated to the $\Lambda$-$\Lambda$-bimodule $\Omega_{{\rm nc}, \Lambda/E}$, which is  naturally $\mathbb{Z}$-graded. We mention that Leavitt rings are also introduced in \cite{CO}, with similar ideas traced back to \cite{Leav}; compare \cite{AAP, AMP, CFHL}.

By \cite{CWKW} and Proposition~\ref{prop:iso-graded}, the zeroth component $L_0$  of $L$ is isomorphic to the colimit of the following sequence of ring homomorphisms.
$$\Lambda\longrightarrow {\rm End}_\Lambda(\Omega_{{\rm nc}, \Lambda/E})^{\rm op} \longrightarrow {\rm End}_\Lambda(\Omega_{{\rm nc}, \Lambda/E}^{\otimes_\Lambda 2})^{\rm op} \longrightarrow \cdots$$
Here, we denote by ${\rm End}_\Lambda(-)$ the endomorphism ring in the category of left $\Lambda$-modules,  the leftmost homomorphism is induced by the right $\Lambda$-action on $\Omega_{{\rm nc}, \Lambda/E}$ and the remaining ones are induced by the tensor endofunctor $(\Omega_{{\rm nc}, \Lambda/E})\otimes_\Lambda -$.

The main result, Theorem~\ref{thm:main}, implies that the singularity category is even equivalent to the stable module category of an FC ring in the sense of \cite{Dam}. It  justifies the title.

 FC rings are coherent analogues of quasi-Frobenius rings. Namely, FC rings are coherent, and noetherian FC rings coincide with quasi-Frobenius rings. By \cite{Dam} and  Lemma~\ref{lem:FC}, the stable category of finitely presented modules over any FC ring is canonically triangulated.

\vskip 5pt
\noindent {\bf Theorem~B}.  \emph{The category $\mathcal{S}$ is equivalent to $L_0\mbox{-{\rm mod}}$, the category of finitely presented left $L_0$-modules. Consequently, $L_0$ is an FC ring, and we have a triangle equivalence
$$\mathbf{D}_{\rm sg}(\Lambda)\simeq L_0\mbox{-\underline{\rm mod}}.$$
}

We mention previous encounters \cite{Smi, CY, CWKW} between singularity categories of artinian rings and a class of Leavitt rings, namely Leavitt path algebras. However, the encounter in Theorem~B is completely different.

The FC ring $L_0$ above is usually non-noetherian, thus not quasi-Frobenius; see Remark~\ref{rem:final}. This gives rise to a new family of non-quasi-Frobenius FC rings, in a somewhat unexpected manner. It turns out that the non-regularity of $L_0$ detects the homological singularity of $\Lambda$; see Corollary~\ref{cor:regular}. We refer to Section~6 for an explicit example where $\Lambda$ is an algebra with radical square zero.  

The key ingredient in the proof of Theorem~B is Proposition~\ref{prop:iso-graded}, which claims that the Leavitt ring \cite{CWKW} is isomorphic to a certain orbit ring \cite{Len, Berg} appearing in the stabilization of any module category.

Let us describe the content of the paper.  In Section~2, we characterize the category of finitely presented graded modules via the existence of cokernels and progenerators. Section~3 is devoted to the study of stabilizations of additive categories. In Section~4, we prove that the stabilization of a module category  is equivalent to the category of  finitely presented graded modules over the Leavitt ring; see Theorem~\ref{thm:s-gr}. In Section~5 we prove Proposition~A and Theorem~B. In Section~6 we provide an explicit description of the Leavitt ring of an algebra with radical square zero.

By default, rings will mean unital rings, and modules will mean left unital modules. For two objects $X$ and $Y$ in  a category $\mathcal{C}$, the  Hom-set ${\rm Hom}_\mathcal{C}(X, Y)$ will be denoted by $\mathcal{C}(X, Y)$. For a set $\mathcal{U}$ of objects in an additive category, we denote by ${\rm add}(\mathcal{U})$ the full subcategory consisting of direct summands of finite direct sums of objects from $\mathcal{U}$.

\section{Categories having cokernels and progenerators}

In this section, we recall basic facts on additive categories having cokernels and progenerators. We characterize the category of finitely presented graded modules using the existence of cokernels and $\Sigma$-progenerators; see Proposition~\ref{prop:equiv-graded}.

Let $\mathcal{C}$ be an additive category. We say that $\mathcal{C}$ \emph{has cokernels} if each morphism in $\mathcal{C}$ has a cokernel. A sequence  $X\stackrel{f}\rightarrow Y \stackrel{g} \rightarrow Z\rightarrow 0$ in $\mathcal{C}$ is  \emph{right-exact} if $g$ is a cokernel of $f$.

Assume that $\mathcal{C}$ has cokernels. An object $P$ is \emph{projective} if for any right-exact sequence $X \rightarrow Y  \rightarrow Z\rightarrow 0$ in $\mathcal{C}$, the induced sequence of abelian groups
$$\mathcal{C}(P, X)\longrightarrow \mathcal{C}(P, Y)  \longrightarrow \mathcal{C}(P, Z)\longrightarrow 0$$
is exact. By a \emph{progenerator} in $\mathcal{C}$, we mean a projective object $P$ such that each object $X$ in $\mathcal{C}$ fits into some right-exact sequence $P_0 \rightarrow P_1  \rightarrow X\rightarrow 0$ with each $P_i\in {\rm add}(P)$.

Let $R$ be an arbitrary ring. Denote by $R\mbox{-mod}$ the category of finitely presented left $R$-modules, and by $R\mbox{-proj}$ its full subcategory of finitely generated projective modules.

The following result is standard; compare \cite[Chapter~III, Section~3]{Aus}.

\begin{lem}\label{lem:ungraded-equiv}
    Let $\mathcal{C}$ be an additive category. Then $\mathcal{C}$ has cokernels and a progenerator if and only if there is an equivalence $\mathcal{C} \simeq R\mbox{-}{\rm mod}$ for some ring $R$.
\end{lem}

\begin{proof}
    The ``if" part is trivial, since  $R\mbox{-mod}$ has cokernels and the regular module $R$ is a progenerator.

    For the ``only if" part, we assume that $\mathcal{C}$ has cokernels and a progenerator $P$. Set $R=\mathcal{C}(P, P)^{\rm op}$ to be  the opposite ring of the endomorphism ring of $P$. The following functor
   $$\mathcal{C}(P, -)\colon \mathcal{C} \longrightarrow R\mbox{-mod}$$
   is well defined. Here, we observe that $\mathcal{C}(P, X)$ is a right $\mathcal{C}(P, P)$-module, and thus a left $R$-module.  The proof of the fully-faithfulness and denseness of $\mathcal{C}(P, -)$ is standard, which is presented as follows.

Since $P$ is a generator, it follows that $\mathcal{C}(P, -)$ is faithful.   For its fullness, we first observe by \cite[Chapter~II, Proposition~2.1]{ARS} that it induces an equivalence
\begin{align}\label{proj}
    {\rm add}(P)\simeq R\mbox{-proj}.
\end{align}
Here, we use implicitly the fact that $\mathcal{C}$ has split idempotents. Take two objects $X$ and $Y$ in $\mathcal{C}$  and any $R$-module homomorphism $u\colon \mathcal{C}(P, X)\rightarrow \mathcal{C}(P, Y)$. Since $P$ is a progenerator, we take two right-exact sequences
   $$P_1\stackrel{f} \longrightarrow P_0 \stackrel{g} \longrightarrow X\longrightarrow 0, \mbox{  and } Q_1\stackrel{f'} \longrightarrow Q_0 \stackrel{g'} \longrightarrow Y\longrightarrow 0$$
with $P_i, Q_i\in {\rm add}(P)$. Applying $\mathcal{C}(P, -)$ to these sequences, we obtain projective presentations for  $\mathcal{C}(P, X)$ and $\mathcal{C}(P, Y)$. Then the given homomorphism $u$ lifts  to a  commutative diagram of $R$-modules.
\begin{align}\label{diag:C}
\xymatrix{
\mathcal{C}(P, P_1) \ar[rr]^-{\mathcal{C}(P, f)} \ar@{.>}[d]_-{u_1} && \mathcal{C}(P, P_0) \ar@{.>}[d]^-{u_0} \ar[rr]^-{\mathcal{C}(P, g)} && \mathcal{C}(P, X) \ar[d]^-{u}\ar[rr] && 0\\
\mathcal{C}(P, Q_1) \ar[rr]^-{\mathcal{C}(P, f')} && \mathcal{C}(P, Q_0) \ar[rr]^-{\mathcal{C}(P, g')} && \mathcal{C}(P, Y) \ar[rr] && 0
}
\end{align}
By the equivalence (\ref{proj}), there are unique morphisms $h_i\colon P_i\rightarrow Q_i$ satisfying $\mathcal{C}(P, h_i)=u_i$; moreover, we have $h_0\circ f=f'\circ h_1$. By the universal property of the cokernel, there is a unique morphism $h\colon X\rightarrow Y$ making the following diagram in $\mathcal{C}$ commute.
\[\xymatrix{
P_1\ar[rr]^-{f} \ar[d]_-{h_1} && P_0 \ar[d]^-{h_0} \ar[rr]^-{g} && X \ar@{.>}[d]^-{h} \ar[rr] && 0\\
Q_1\ar[rr]^-{f'} && Q_0 \ar[rr]^-{g'} && Y \ar[rr] && 0
}\]
Applying $\mathcal{C}(P, -)$ to the diagram above and comparing the resulting diagram with (\ref{diag:C}), we infer that $u=\mathcal{C}(P, h)$, as required.

To show that $\mathcal{C}(P, -)$ is dense, we take any finitely presented $R$-module $M$ with a projective presentation $E_1\stackrel{v}\rightarrow E_0\rightarrow M\rightarrow 0$. By the equivalence (\ref{proj}), we may assume that $E_i=\mathcal{C}(P, P_i)$ for $P_i\in {\rm add}(P)$; moreover, there is a unique morphism $k\colon P_1\rightarrow P_0$ satisfying $\mathcal{C}(P, k)=v$. Set $C$ to be the cokernel of $k$. It follows that $\mathcal{C}(P, C)$ is isomorphic to $M$.
\end{proof}

Let $\Sigma$ be an automorphism on $\mathcal{C}$. By a \emph{$\Sigma$-progenerator} in $\mathcal{C}$, we mean a projective object $P$ such that each object $X$ in $\mathcal{C}$ fits into some right-exact sequence $P_0 \rightarrow P_1  \rightarrow X\rightarrow 0$ with each $P_i\in {\rm add}\{\Sigma^n(P)\; |\; n\in \mathbb{Z}\}$.

Let $\Gamma=\oplus_{n\in \mathbb{Z}} \Gamma_n$ be a $\mathbb{Z}$-graded ring. A graded $\Gamma$-module  is usually written as $M=\oplus_{n\in \mathbb{Z}} M_n$. For each integer $p$, we have the \emph{degree-shifted} module $M(p)$, whose underlying $\Gamma$-module is still $M$ and which is graded such that $M(p)_n=M_{p+n}$.  Denote by $\Gamma\mbox{-grmod}$ the category of finitely presented left graded $\Gamma$-modules. For each $p$, we denote by $(p)\colon \Gamma\mbox{-grmod}\rightarrow \Gamma\mbox{-grmod}$ the degree-shift automorphism. We refer to \cite[Chapter~2]{NVO} for details.

We have a variation of Lemma~\ref{lem:ungraded-equiv}.

\begin{prop}\label{prop:equiv-graded}
    Let $\mathcal{C}$ be an additive category with an automorphism $\Sigma$. Then $\mathcal{C}$ has cokernels and an $\Sigma$-progenerator if and only if there is an equivalence $\Phi\colon \mathcal{C} \simeq \Gamma\mbox{-}{\rm grmod}$ for some $\mathbb{Z}$-graded ring $\Gamma$ satisfying $(1)\Phi\simeq \Phi\Sigma$.
\end{prop}

\begin{proof}
    The ``if" part is trivial, since $\Gamma\mbox{-}{\rm grmod}$ has cokernels and the regular module $\Gamma$ is a $(1)$-progenerator.

    For the ``only if" part, we take a $\Sigma$-progenerator $P$. Following \cite[Section~2]{Len} and \cite[Section~2]{Berg}, we define the (opposite) \emph{orbit ring} of $\Sigma$ at $P$ to be
    \begin{align}\label{defn:gamma}
        \Gamma=\Gamma(P; \Sigma)=\bigoplus_{n\in \mathbb{Z}} \mathcal{C}(P, \Sigma^n(P))^{\rm op},
    \end{align}
which is naturally $\mathbb{Z}$-graded.   For $f\in \mathcal{C}(P, \Sigma^n(P))$ and $g\in \mathcal{C}(P, \Sigma^m(P))$, their product $fg$ in $\Gamma$ is given by $\Sigma^n(g)\circ f$. In particular, the unit of $\Gamma$ is given by ${\rm Id}_P$, the identity endomorphism of $P$.

    We have a well-defined functor
    $$\Phi=\bigoplus_{n\in \mathbb{Z}} \mathcal{C}(P, \Sigma^n(-))\colon \mathcal{C}\longrightarrow \Gamma\mbox{-}{\rm grmod}.$$
    We observe that $\Phi(X)(1)=\Phi(\Sigma(X))$ for each object $X$. The same argument in Lemma~\ref{lem:equiv} shows that $\Phi$ is fully faithful and dense.
\end{proof}

Recall that a $\mathbb{Z}$-graded ring $\Gamma=\oplus_{n\in \mathbb{Z}} \Gamma_n$ is called \emph{strongly graded} provided that $\Gamma_n \Gamma_m=\Gamma_{n+m}$ for all $n, m\in \mathbb{Z}$; see \cite{Dade}, \cite[Chapter~3]{NVO} and \cite[Section~1.5]{Hazrat}.  In this situation, we have  a well-known equivalence
\begin{align}\label{equ:Dade}
 \Gamma\mbox{-grmod} \stackrel{\sim}\longrightarrow \Gamma_0\mbox{-mod},
\end{align}
which sends a graded $\Gamma$-module $M$ to its zeroth homogeneous component $M_0$; see \cite[Theorem~2.8]{Dade} and \cite[Theorem~1.5.1]{Hazrat}.

\begin{lem}\label{lem:equiv}
   Let $\mathcal{C}$ be an additive category with an automorphism $\Sigma$. Assume further that it has cokernels and an $\Sigma$-progenerator $P$. Then the orbit ring $\Gamma(P; \Sigma)$ is strongly graded if and only if ${\rm add}(P)={\rm add}(\Sigma(P))$.

   When these equivalent conditions hold, we have an equivalence
   $$\mathcal{C}(P, -)\colon \mathcal{C}\stackrel{\sim}\longrightarrow \Gamma(P; \Sigma)_0\mbox{-}{\rm mod}.$$
\end{lem}

\begin{proof}
  For the ``if" part,  we  assume that ${\rm add}(P)={\rm add}(\Sigma(P))$. Since $P$ belongs to ${\rm add}(\Sigma(P))$,   there is a finite index set $I$ with morphisms $f_i\colon P\rightarrow \Sigma(P)$ and $g_i\colon \Sigma(P)\rightarrow P$ satisfying ${\rm Id}_P=\sum_{i\in I} g_i\circ f_i$. In the orbit ring $\Gamma=\Gamma(P;  \Sigma)$, we have
  $$1_\Gamma=\sum_{i\in I} f_i\Sigma^{-1}(g_i).$$
  It follows that $1_\Gamma\in \Gamma_{1}\Gamma_{-1}$. Then we have $\Gamma_0=\Gamma_{1}\Gamma_{-1}$, since $\Gamma_{1}\Gamma_{-1}$ is a two-sided ideal of $\Gamma_0$. Similarly, one proves $\Gamma_0=\Gamma_{-1}\Gamma_{1}$. In view of \cite[Proposition~1.6]{Dade},  we conclude  that $\Gamma$ is strongly graded.

  For the ``only if" part, it suffices to observe the following well-known fact: for any strongly graded ring $\Gamma$, we always have ${\rm add}(\Gamma)={\rm add}(\Gamma(1))$ in $\Gamma\mbox{-grmod}$. Then we apply the equivalence in Proposition~\ref{prop:equiv-graded}.

  For the last statement, we just combine the equivalence in Proposition~\ref{prop:equiv-graded} with (\ref{equ:Dade}). Alternatively, we deduce the equivalence by Lemma~\ref{lem:equiv}, since $P$ is a progenerator of $\mathcal{C}$ in this situation.
\end{proof}

\begin{rem}
   Proposition~\ref{prop:equiv-graded} and Lemma~\ref{lem:equiv} still hold when $\Sigma$ is an autoequivalence. Indeed, one might use the stabilization \cite{Hel68, Tie} to replace $\Sigma$ by an automorphism.
\end{rem}

\section{The stabilization of a looped category}

In this section, we prove in Proposition~\ref{prop:stab} that the stabilization of a category with cokernels and a progenerator always has cokernels and an $\Sigma$-progenerator. For stabilizations, we refer to \cite[Chapter~I]{Hel68} and \cite[Section~1]{Tie}.

Let us first recall generalities on stabilizations. By a \emph{looped category} \cite{Bel}, we mean a pair $(\mathcal{C}, \Omega)$ consisting of a category $\mathcal{C}$ and an endofunctor $\Omega$ on $\mathcal{C}$. The looped category is called \emph{stable} if $\Omega$ is an autoequivalence,  and called \emph{strictly stable} if $\Omega$ is an automorphism.

Let $(\mathcal{C}, \Omega)$ and $(\mathcal{D}, \Delta)$ be two looped categories. A \emph{looped functor}
$$(F, \delta)\colon (\mathcal{C}, \Omega)\longrightarrow (\mathcal{D}, \Delta)$$
consists of a functor $F\colon \mathcal{C}\rightarrow \mathcal{D}$ and a natural isomorphism $\delta\colon F\Omega \rightarrow \Delta F$.

Let $(\mathcal{C}, \Omega)$ be a looped category. We define a new category $\mathcal{S}=\mathcal{S}(\mathcal{C}, \Omega)$ as follows. The objects are given by pairs $(X, n)$, which consist of an object $X$ in $\mathcal{C}$ and an integer $n$. The Hom-set from $(X, n)$ to $(Y, m)$ is given by a colimit
$$\mathcal{S}((X, n), (Y, m))={\rm colim}\; \mathcal{C}(\Omega^{p-n}(X), \Omega^{p-m}(Y)),$$
where $p$ runs over all integers satisfying $p\geq {\rm max}\{n, m\}$, and the structure map is induced by the endofunctor $\Omega$. For a morphism $f\in \mathcal{C}(\Omega^{p-n}(X), \Omega^{p-m}(Y))$, its image in $\mathcal{S}((X, n), (Y, m))$ will be denoted by $[f; p]\colon (X, n)\rightarrow (Y, m)$. By the very definition, we have
\begin{align}\label{equ:omega}
    [f; p]=[\Omega^k(f); k+p]
\end{align}
for any $k\geq 0$.

The composition in $\mathcal{S}$ is induced by the one in $\mathcal{C}$. To be more precise, we take any morphism $[g; q]\colon (Y, m)\rightarrow (Z, l)$. Since $[g;q]=[\Omega^k(g); q+k]$ for all $k\geq 0$, we may assume that $q\geq p$. We define the composition by
$$[g;q]\circ [f; p]=[g\circ \Omega^{q-p}(f); q].$$

\begin{lem}\label{lem:stab}
    Consider the morphism $[f; p]\colon (X, n)\rightarrow (Y, m)$ above. Then we have a commutative diagram
    \[\xymatrix{
    (X, n) \ar[d]_-{[{\rm Id}_{\Omega^{p-n}(X)}; p]} \ar[rr]^-{[f; p]} && (Y, m) \ar[d]^-{[{\rm Id}_{\Omega^{p-m}(Y)}; p]}\\
    (\Omega^{p-n}(X), p) \ar[rr]^-{[f;p]} && (\Omega^{p-m}(Y), p)
    }\]
    with vertical morphisms being isomorphisms in $\mathcal{S}$.
\end{lem}

\begin{proof}
    The commutativity is trivial. The inverse of the given morphism $[{\rm Id}_{\Omega^{p-n}(X)}; p]$ is given by $[{\rm Id}_{\Omega^{p-n}(X)}; p]\colon (\Omega^{p-n}(X),  p)\rightarrow (X, n)$.
\end{proof}

\begin{rem}\label{rem:enlarge}
By the vertical isomorphism above, we may always enlarge the second entry of an object in $\mathcal{S}$. Moreover, two objects $(A, n)$ and $(B, m)$ are isomorphic if and only if there is an isomorphism $\Omega^{p-n}(A)\simeq \Omega^{p-m}(B)$ for some $p\geq {\rm max}\{n, m\}$.
\end{rem}

The category $\mathcal{S}$ carries an automorphism $\Sigma$ defined by $\Sigma(X, n)=(X, n+1)$ on objects. It sends $[f;p]\colon (X, n)\rightarrow (Y, m)$ to $[f; p+1]\colon (X, n+1)\rightarrow (Y, m+1)$ on morphisms. Consequently, we have a strictly stable category $(\mathcal{S}, \Sigma^{-1})$.

We have a functor $\mathbf{S}\colon \mathcal{C}\rightarrow \mathcal{S}$,  which sends $X$ to $(X, 0)$, and sends a morphism $f$ to $[f;0]$. For each object $X$, we have a natural isomorphism
$$\theta_X=[{\rm Id}_{\Omega(X)}; 0]\colon \mathbf{S}\Omega(X)=(\Omega(X), 0) \stackrel{\sim}\longrightarrow (X, -1)=\Sigma^{-1}\mathbf{S}(X).$$
In other words, we have a looped functor
$$(\mathbf{S}, \theta)\colon (\mathcal{C}, \Omega)\longrightarrow (\mathcal{S}, \Sigma^{-1})$$
from a looped category to a strictly stable category. This process is called the \emph{stabilization} of $(\mathcal{C}, \Omega)$ by formally inverting $\Omega$.

In what follows, we assume that $\mathcal{C}$ is an additive category and that the endofunctor $\Omega$ is additive. It follows that $\mathcal{S}=\mathcal{S}(\mathcal{C}, \Omega)$ is also additive.

For an additive full subcategory $\mathcal{X}$ of $\mathcal{C}$, we denote by $\mathcal{C}/\mathcal{X}$ the corresponding \emph{factor category} \cite[IV.1]{ARS}. It has the same objects as $\mathcal{C}$ does. The Hom-group $\mathcal{C}/\mathcal{X}(A, B)$ is given by the quotient group $\mathcal{C}(A, B)/\mathcal{I}(A, B)$, where $\mathcal{I}(A, B)$ is the subgroup consisting of morphisms factoring through $\mathcal{X}$. For a morphism $f\colon A\rightarrow B$ in $\mathcal{C}$, its image in $\mathcal{C}/\mathcal{X}(A, B)$ is denoted by $\underline{f}$.

We observe that the stabilization behaves well with respect to factor categories.

\begin{lem} \label{lem:fac}
Assume that $\mathcal{X}\subseteq \mathcal{C}$ is a full additive subcategory satisfying $\Omega(\mathcal{X})\subseteq \mathcal{X}$. Then there is an equivalence
$$\mathcal{S}(\mathcal{C}, \Omega)/{\mathcal{S}(\mathcal{X}, \Omega|_\mathcal{X})}\simeq \mathcal{S}(\mathcal{C}/\mathcal{X}, \Omega).$$
\end{lem}

Here, $\Omega|_\mathcal{X}$ means  the restriction of $\Omega$. Then we have the restricted looped category $(\mathcal{X}, \Omega|_\mathcal{X})$. The stabilization $\mathcal{S}(\mathcal{X}, \Omega|_\mathcal{X})$ is viewed as a full subcategory of $\mathcal{S}(\mathcal{C}, \Omega)$. The endofunctor $\Omega$ on $\mathcal{C}$ induces an endofunctor on the factor category $\mathcal{C}/\mathcal{X}$, which is still denoted by $\Omega$.

\begin{proof}
The canonical functor $\pi\colon \mathcal{S}(\mathcal{C}, \Omega)\rightarrow \mathcal{S}(\mathcal{C}/\mathcal{X}, \Omega)$ is full and dense, which certainly vanishes on $\mathcal{S}(\mathcal{X}, \Omega|_\mathcal{X})$.

It remains to show that any morphism annihilated by $\pi$ necessarily factors through some object in $\mathcal{S}(\mathcal{X}, \Omega|_\mathcal{X})$. By Lemma~\ref{lem:stab}, we might assume that the morphism is of the form $[f; p]\colon (A, p)\rightarrow (B, p)$ with $f\in \mathcal{C}(A, B)$. Then we have $\pi([f; p])=[\underline{f}; p]=0$ in $\mathcal{S}(\mathcal{C}/\mathcal{X}, \Omega)$. It implies that $\Omega^k(\underline{f})=0$ in $\mathcal{C}/\mathcal{X}$ for some $k\geq 0$. In other words, there are morphisms $a\colon X\rightarrow \Omega^k(B)$ and $b\colon \Omega^k(A)\rightarrow X$ with $X\in \mathcal{X}$ satisfying $\Omega^k(f)=a\circ b$. These morphisms give rise to two morphisms $[a; k+p]\colon (X, k+p)\rightarrow (B, p)$ and $[b; k+p]\colon (A, p)\rightarrow (X, k+p)$ in $\mathcal{S}(\mathcal{C}, \Omega)$.  By (\ref{equ:omega}), we have a desire factorization
$$[f;p]=[\Omega^k(f); k+p]=[a; k+p]\circ [b; k+p].$$
Since the object $(X, k+p)$ belongs to $\mathcal{S}(\mathcal{X}, \Omega|_\mathcal{X})$, we are done.
\end{proof}

\begin{lem}\label{lem:stab-cok}
    Assume that $\mathcal{C}$ has cokernels and that $\Omega$ is right-exact. Then $\mathcal{S}$ has cokernels.
\end{lem}

Here, the right-exactness of $\Omega$ means that it sends right-exact sequences in $\mathcal{C}$  to right-exact sequences in $\mathcal{C}$.

\begin{proof}
We have to show that any morphism in $\mathcal{S}$ has a cokernel.    By Lemma~\ref{lem:stab}, we may assume that the morphism is of the form $[f; p]\colon (A, p)\rightarrow (B,  p)$ with $f\in \mathcal{C}(A, B)$. Assume that $g\colon B\rightarrow C$ is a cokernel of $f$. We claim that $[g;p]\colon (B, p)\rightarrow (C, p)$ is a cokernel of $[f;p]$.

We first show that $[g;p]$ is epic. Take any morphism $[t;q]\colon (C, p)\rightarrow (Z, l)$ with $q\geq {\rm max}\{p, l\}$ and $t\colon \Omega^{q-p}(C)\rightarrow \Omega^{q-l}(Z)$. We assume that
$$[t;q]\circ [g;p]=[t\circ \Omega^{q-p}(g); q]=0.$$
It implies that $\Omega^k(t\circ \Omega^{q-p}(g))=0$ in $\mathcal{C}$ for some $k\geq 0$. Since $\Omega^{k+q-p}(g)$ is epic, we have $\Omega^k(t)=0$, which implies that $[t; q]=0$.

It remains to show that any morphism $[h;q]\colon (B,p)\rightarrow (Z, l)$ satisfying $[h;q]\circ [f;p]=0$ factors through $[g; p]$. Here, we have $q\geq {\rm max}\{p, l\}$ and $h\colon \Omega^{q-p}(B)\rightarrow \Omega^{q-l}(Z)$. The vanishing condition implies that
$$0=\Omega^k(h\circ \Omega^{q-p}(f))=\Omega^k(h)\circ \Omega^{k+q-p}(f)$$
for some $k\geq 0$. Since $\Omega^{k+q-p}(g)$ is a cokernel of $\Omega^{k+q-p}(f)$, there is a unique morphism $u\colon \Omega^{k+q-p}(C)\rightarrow \Omega^{k+q-l}(Z)$ satisfying
$$u\circ \Omega^{k+q-p}(g)=\Omega^k(h).$$
We have a morphism $[u; k+q]\colon (C, p)\rightarrow (Z,l)$. Using (\ref{equ:omega}) twice, we have
$$[h;q]=[\Omega^{k}(h); k+q]=[u; k+q]\circ [\Omega^{k+q-p}(g); k+q]=[u; k+q]\circ [g; p].$$
This completes the proof.
\end{proof}

\begin{lem}\label{lem:abel}
    Assume that $\mathcal{C}$ is an abelian category and the endofunctor $\Omega$ is exact. Then the stabilization $\mathcal{S}$ is also an abelian category.
\end{lem}

\begin{proof}
    The same proof in Lemma~\ref{lem:stab-cok} shows that $\mathcal{C}$ has kernels; furthermore, any morphism $[f; p]\colon (A, p)\rightarrow (B, p)$ admits a canonical factorization
    $$(A, p) \stackrel{[\pi; p]}\longrightarrow ({\rm Im}(f), p) \stackrel{[\iota; p]} \longrightarrow (B, p),$$
    where $f=\iota\circ \pi$ is a canonical factorization of $f$ in $\mathcal{C}$. It follows that $\mathcal{S}$ is abelian.
\end{proof}

\begin{rem}\label{rem:mono}
    Keep the assumptions in Lemma~\ref{lem:abel}. We observe that up to isomorphism, any monomorphism in $\mathcal{S}$ is of the form $[f; p] \colon (A, p)\rightarrow (B, p)$ with $f\colon A\rightarrow B$ a monomorphism in $\mathcal{C}$.
\end{rem}

\begin{prop}\label{prop:stab}
    Assume that $\mathcal{C}$ has cokernels and a progenerator $P$, and that $\Omega$ is right-exact and satisfying $\Omega(P)\in {\rm add}(P)$. Then $\mathcal{S}=\mathcal{S}(\mathcal{C}, \Omega)$ has cokernels and an $\Sigma$-progenerator $\mathbf{S}(P)=(P,0)$.
\end{prop}

\begin{proof}
    By Lemma~\ref{lem:stab-cok}, $\mathcal{S}$ has cokernels. Moreover, by its proof, any right-exact sequence in $\mathcal{S}$ is isomorphic to
    $$\eta\colon (A, p)\stackrel{[f;p]}\longrightarrow (B, p) \stackrel{[g;p]} \longrightarrow (C, p)\longrightarrow 0,$$
    where $g\colon B\rightarrow C$ is a cokernel of $f\colon A\rightarrow B$ in $\mathcal{C}$.

    To prove that $(P, 0)$ is projective, it suffices to show the exactness of $\mathcal{S}((P, 0), \eta)$. This sequence is a direct colimit  of
    $$\eta^j\colon \mathcal{C}(\Omega^j(P), \Omega^{j-p}(A)) \longrightarrow \mathcal{C}(\Omega^j(P), \Omega^{j-p}(B))  \longrightarrow \mathcal{C}(\Omega^j(P), \Omega^{j-p}(C)) \longrightarrow 0.$$
    Here, $j$ runs over all integers satisfying $j\geq {\rm max}\{0, p\}$. Since $\Omega^j(P)$ is projective and $\Omega$ is right-exact, these sequences $\eta^j$ are all exact. This implies the required exactness.

    Take any object $(X, n)$ and a projective presentation $P_1\rightarrow P_0\rightarrow X\rightarrow 0$ with $P_i\in {\rm add}(P)$. Then we have an induced projective presentation
    $$(P_1, n) \longrightarrow (P_0, n) \longrightarrow (X, n) \longrightarrow 0.$$
    Recall that $(P_i, n)=\Sigma^n(P_i)$, which belongs to ${\rm add}\{\Sigma^j(P, 0)\; |\; j\in \mathbb{Z}\}$. This implies that $(P, 0)$ is an $\Sigma$-progenerator of $\mathcal{S}$.
\end{proof}

Recall from (\ref{defn:gamma}) the orbit ring  $\Gamma(\mathbf{S}(P); \Sigma)$  of $\Sigma$ at $\mathbf{S}(P)=(P, 0)$.

\begin{lem}\label{lem:ss}
    Keep the assumptions above. Assume further that the descending chain in $\mathcal{C}$
    $${\rm add}(P)\supseteq {\rm add}(\Omega(P)) \supseteq {\rm add}(\Omega^2(P))\supseteq \cdots$$ stabilizes. Then we have ${\rm add}(\mathbf{S}(P))={\rm add}(\Sigma \mathbf{S}(P))$. Consequently, the orbit ring $\Gamma(\mathbf{S}(P); \Sigma)$ is strongly graded.
\end{lem}

\begin{proof}
We recall that $\mathbf{S}(P)=(P, 0)$ and $\Sigma \mathbf{S}(P)=(P, 1)$. Since $\mathbf{S}(P)$ is isomorphic to $(\Omega(P), 1)$, which belongs to ${\rm add}(\Sigma \mathbf{S}(P))$.  It follows that ${\rm add}(\mathbf{S}(P))\subseteq {\rm add}(\Sigma \mathbf{S}(P))$.

Let us prove ${\rm add}(\Sigma \mathbf{S}(P))\subseteq {\rm add}(\mathbf{S}(P))$. We assume that
$${\rm add}(\Omega^{n_0}(P))={\rm add}(\Omega^{n_0+1}(P))$$
for some $n_0\geq 0$. Take any object $(X, n)$ in ${\rm add}(\Sigma \mathbf{S}(P))$. Assume that
$$(X, n)\oplus (Y, n')\simeq (P, 1)^{\oplus l}$$
for some $l\geq $. In view of Remark~\ref{rem:enlarge}, we might assume $n'=n$. It follows that there is a sufficiently large $p$ such that
$$\Omega^{p-n}(X\oplus Y)\simeq \Omega^{p-1}(P)^{\oplus l}.$$
We will assume that $p$ is larger than $n_0+1$. It follows that ${\rm add}(\Omega^{p-1}(P))={\rm add}(\Omega^{p}(P))$. Then $\Omega^{p-n}(X\oplus Y)$ belongs to ${\rm add}(\Omega^{p}(P))$. Assume that
$$\Omega^{p-n}(X\oplus Y)\oplus Z\simeq \Omega^{p}(P)^{\oplus l'}$$
for some $l'\geq 1$. It follows that
$$(X, n)\oplus (Y, n)\oplus (Z, p)\simeq (P, 0)^{\oplus l'}.$$
In particular, the object $(X, n)$ belongs to ${\rm add}(\mathbf{S}(P))$, which implies ${\rm add}(\Sigma \mathbf{S}(P))\subseteq {\rm add}(\mathbf{S}(P))$. This proves the required equality. The last statement follows from Lemma~\ref{lem:equiv}.
\end{proof}

\section{Leavitt rings and the stabilizations of module categories}

In this section, we prove in Theorem~\ref{thm:s-gr} that the stabilization of a module category by formally inverting a tensor endofunctor is equivalent to the category of graded modules over the Leavitt ring \cite{CWKW}.

Let $R$ be an arbitrary ring, and $_RM_R$ an $R$-$R$-module such that its underlying  left $R$-module $_RM$ is finitely generated projective. Then we have a looped category $(R\mbox{-mod}, M\otimes_R-)$, and form the stabilization $\mathcal{S}=\mathcal{S}(R\mbox{-mod}, M\otimes_R-)$.

Denote by $M^*={\rm Hom}_R({_R}M, {_R}R)$ the \emph{left-dual bimodule}. Its $R$-$R$-bimodule structure is given such that
$$(af)(x)=f(xa) \mbox{ and } (fb)(x)=f(x)b$$
for any $a, b\in R$, $x\in M$ and $f\in M^*$. Since $_RM$ is finitely generated projective, it has a dual basis $\{\alpha_j, \alpha^*_j\}_{j\in J}$, where $J$ is a finite index set, $\alpha_j\in M$ and $\alpha_j^*\in M^*$ such that $\alpha_i^*(\alpha_j) = \delta_{i,j} 1_R$ for any $i, j \in J$. Here, $\delta$ denotes the Kronecker delta. These data satisfy
\begin{align}\label{equ:dual-basis}
x=\sum_{j\in J}\alpha_j^*(x)\alpha_j \mbox{ and } f= \sum_{j\in J} \alpha_j^* f(\alpha_j)
\end{align}
for all $x\in M$ and $f\in M^*$. The \emph{Casimir element}
\begin{align}\label{align:Casimir} 
c=\sum_{j\in J} \alpha_j^*\otimes_R \alpha_j\in M^*\otimes_R M
\end{align}
does not depend on the choice of the dual basis.

For each $p, k\geq 0$, we have a canonical isomorphism of $R$-$R$-bimodules
$$\phi^{p, k}\colon (M^*)^{\otimes_R p}\otimes_R M^{\otimes_R k} \longrightarrow {\rm Hom}_R(M^{\otimes_R p }, M^{\otimes_R k}),$$
which sends $f_{1, p}\otimes_R a_{1, k}$ to the map
$$b_{1, p}\longmapsto f_p(b_1 \cdots f_2(b_{p-1}f_1(b_p))\cdots ) a_{1, k}.$$
Here and later, we write $f_{1, p}=f_1\otimes_R \cdots \otimes_R f_p$, $a_{1,k}=a_1\otimes_R \cdots \otimes_R a_k$ and $b_{1, p}=b_1\otimes_R \cdots \otimes_R b_p$, which are typical tensors.  In particular, $\phi^{0, 0}\colon R\rightarrow {\rm Hom}_R(R, R)$ sends an element $a$ to the map $(b\mapsto ba)$.

The verification of the following compatibility result is straightforward.

\begin{lem}\label{lem:compo}
For $p, k, j\geq 0$, we have
$$\phi^{k, j}(g_{1, k}\otimes_R b_{1, j}) \circ \phi^{p, k}(f_{1, p}\otimes_R a_{1, k})=\phi^{p, j}(f_{1, p}g_k(a_1\cdots g_2(a_{k-1}g_1(a_k))\cdot \cdot \cdot) \otimes_R b_{1,j}),$$
with  $g_{1, k}=g_1\otimes_R \cdots \otimes_R g_k\in (M^*)^{\otimes_R k}$ and $b_{1, j}=b_1\otimes_R \cdots \otimes_R b_j\in M^{\otimes_R j}$. \hfill $\square$
\end{lem}

For each $p, k\geq 0$, we have an $R$-$R$-bimodule homomorphism
$$\Delta^{p, k}\colon (M^*)^{\otimes_R p}\otimes_R M^{\otimes_R k} \longrightarrow (M^*)^{\otimes_R p+1}\otimes_R M^{\otimes_R k+1}, $$
which sends $f_{1, p}\otimes_R a_{1, k}$ to $f_{1, p}\otimes_R c\otimes_R a_{1, k}$ with $c$ the Casimir element \eqref{align:Casimir}.

\begin{lem}\label{lem:compare}
We have a commutative diagram of $R$-$R$-bimodules.
\[\xymatrix{
(M^*)^{\otimes_R p}\otimes_R M^{\otimes_R k}\ar[d]_-{\Delta^{p, k}} \ar[rr]^-{\phi^{p, k}} && {\rm Hom}_R(M^{\otimes_R p }, M^{\otimes_R k})\ar[d]^-{M\otimes_R-}\\
(M^*)^{\otimes_R p+1}\otimes_R M^{\otimes_R k+1}\ar[rr]^-{\phi^{p+1, k+1}} &&   {\rm Hom}_R(M^{\otimes_R p+1 }, M^{\otimes_R k+1})
}\]
Here, the vertical map on the right side is induced by the endofunctor $M\otimes_R-$.
\end{lem}

\begin{proof}
     It suffices to observe that the composite map $\phi^{p+1, k+1}\circ \Delta^{p, k}$ sends $f_{1, p}\otimes_R a_{1, k}$ to the following map
    \begin{align*}
    b_{1, p+1}\longmapsto &\sum_{j\in J} \alpha^*_j(b_1 \cdots f_2(b_{p}f_1(b_{p+1}))\cdots ) \alpha_j\otimes_R a_{1, k}\\
    &=b_1\otimes_R f_p(b_2\cdots f_2(b_{p}f_1(b_{p+1}))\cdots ) a_{1, k}.
    \end{align*}
    Here, we use (\ref{equ:dual-basis}).
\end{proof}

Recall from \cite[Definition~2.4]{CWKW} the \emph{Leavitt ring} associated to the $R$-$R$-bimodule $M$ is defined to be
$$L_R(M)=T_R(M^*\oplus M)/{(x\otimes_R f-f(x), c-1_R \; |\; x\in M, f\in M^*)}.$$
Here, $T_R(M^* \oplus M)$ denotes the tensor ring of the bimodule $M^*\oplus M$. The Leavitt ring is $\mathbb{Z}$-graded such that ${\rm deg}(x)=-1$ and ${\rm deg}(f)=1$ for $x\in M$ and $f\in M^*$. We mention that Leavitt rings are certain Cuntz-Primsner rings in the sense of \cite[Definition~3.16]{CO}; see also \cite[Section~2]{CFHL}.

We have the following key isomorphism, which realizes the Leavitt ring $L_R(M)$ as an orbit ring in the stabilization $\mathcal{S}=\mathcal{S}(R\mbox{-mod}, M\otimes_R-)$.

\begin{prop}\label{prop:iso-graded}
    There is an isomorphism of $\mathbb{Z}$-graded rings
    $$L_R(M)\simeq \Gamma(\mathbf{S}(R); \Sigma).$$
\end{prop}

Here, we recall from Proposition~\ref{prop:stab} that $\mathbf{S}(R)=(R, 0)$ is an $\Sigma$-progenerator in $\mathcal{S}$, and that $\Gamma(\mathbf{S}(R); \Sigma)$ is the orbit ring (\ref{defn:gamma}) of $\Sigma$ at $\mathbf{S}(R)$.

\begin{proof}
Recall that  $T_R(M)=\bigoplus_{k\geq 0} M^{\otimes_R k}$. Consider the following homomorphism of $R$-$R$-bimodules.
$$\Delta^p=\bigoplus_{k\geq 0} \Delta^{p, k} \colon (M^*)^{\otimes_R p}\otimes_R T_R(M)\longrightarrow (M^*)^{\otimes_R p+1}\otimes_R T_R(M)$$
We form the colimit
\begin{align}\label{colimit:p}
{\rm colim}_{p\geq 0}\; (M^*)^{\otimes_R p}\otimes_R T_R(M).
\end{align}
For each $p\geq 0$, we have a homomorphism of $R$-$R$-bimodules
$$\psi^p\colon \colon (M^*)^{\otimes_R p}\otimes_R T_R(M)\longrightarrow L_R(M),$$
which sends $f_{1, p}\otimes_R a_{1, k}$ to the product $f_{1, p}a_{1, k}$ in $L_R(M)$. By \cite[Theorem~2.6]{CWKW} these homomorphisms induce an isomorphism
$$\Psi\colon {\rm colim}_{p\geq 0}\; (M^*)^{\otimes_R p}\otimes_R T_R(M)\stackrel{\sim}\longrightarrow L_R(M).$$

Write $\Gamma=\Gamma(\mathbf{S}(R); \Sigma)$. Recall that $\Gamma=\bigoplus_{n\in \mathbb{Z}} \mathcal{S}(\mathbf{S}(R),\Sigma^n \mathbf{S}(R))$. Therefore, we have
\begin{align*}
    \Gamma &=\bigoplus_{n\in \mathbb{Z}} \mathcal{S}((R, 0), (R, n))\\
          &= \bigoplus_{n\in \mathbb{Z}} {\rm colim}_{p\geq {\rm max}\{0, n\}} \;  {\rm Hom}_R(M^{\otimes_R p}, M^{\otimes_R (p-n)})\\
          &\simeq  {\rm colim}_{p\geq 0} \; \bigoplus_{k\geq 0} {\rm Hom}_R(M^{\otimes_R p}, M^{\otimes_R k})\\
          &\simeq {\rm colim}_{p\geq 0}\;  \bigoplus_{k\geq 0} (M^*)^{\otimes_R p}\otimes_R M^{\otimes_R k}\\
    &\simeq {\rm colim}_{p\geq 0}\; (M^*)^{\otimes_R p}\otimes_R T_R(M).
\end{align*}
Here, we use the isomorphisms $\phi^{p, k}$. Moreover, by Lemma~\ref{lem:compare}, we infer that the structure maps in the colimit above coincide with $\Delta^p$. In other words, we obtain an isomorphism of $R$-$R$-bimodules
$$\Phi\colon {\rm colim}_{p\geq 0}\; (M^*)^{\otimes_R p}\otimes_R T_R(M) \stackrel{\sim}\longrightarrow \Gamma.$$

It remains to show that the isomorphism  $\Phi\circ \Psi^{-1}\colon L_R(M) \rightarrow \Gamma$ is compatible with the multiplications. For this end, we take two typical elements $f_{1, p}\otimes_R a_{1, k}$ and $g_{1, q}\otimes_R b_{1, j}$ appearing in the colimit (\ref{colimit:p}). Since we can apply $\Delta^{p,k}$ or $\Delta^{q, j}$ to make the tensors longer, we are able to assume that $k=q$.

We have $\Psi(f_{1, p}\otimes_R a_{1, k})=f_1\cdots f_p a_1\cdots a_k$ and $\Psi(g_{1, k}\otimes_R b_{1, j})=g_1\cdots g_k b_1\cdots b_j$. Their product in $L_R(M)$ is given by
\begin{align}\label{equ:inL}
    &f_1\cdots f_p g_k(a_1\cdots g_2(a_2g_1(a_k))\cdots ) b_1\cdots b_j \\
    &= \Psi(f_{1, p}g_k(a_1\cdots g_2(a_{k-1}g_1(a_k)) \cdots) \otimes_R b_{1, j})\in L_R(M). \nonumber
\end{align}
We observe that
$$\Phi(f_{1, p}\otimes_R a_{1, k})=[ \phi^{p, k}(f_{1, p}\otimes_R a_{1, k}); p]\in \Gamma_{p-k}$$
and
$$\Phi(g_{1, k}\otimes_R b_{1, j})=[ \phi^{k, j}(g_{1, k}\otimes_R b_{1, j}); k]\in \Gamma_{k-j}.$$
Their product in $\Gamma$ is given by the following composition in $\mathcal{S}$.
\begin{align*}
&\Sigma^{p-k}([ \phi^{k, j}(g_{1, k}\otimes_R b_{1, j});k]) \circ [ \phi^{p, k}(f_{1, p}\otimes_R a_{1, k});p]\\
&=[ \phi^{k, j}(g_{1, k}\otimes_R b_{1, j});p] \circ [ \phi^{p, k}(f_{1, p}\otimes_R a_{1, k});p]\\
&=[\phi^{k, j}(g_{1, k}\otimes_R b_{1, j})\circ  \phi^{p, k}(f_{1, p}\otimes_R a_{1, k});p]\\
&=[\phi^{p, j} (f_{1, p}g_k(a_1\cdots g_2(a_{k-1}g_1(a_k) \cdots ) \otimes_R b_{1, j}); p]
\end{align*}
Here, the last equality uses Lemma~\ref{lem:compo}. Then we conclude that
$$\Phi(f_{1, p}\otimes_R a_{1, k}) \Phi(g_{1, k}\otimes_R b_{1, j})=\Phi(f_{1, p}g_k(a_1\cdots g_2(a_{k-1}g_1(a_k)) \cdots) \otimes_R b_{1, j})$$
holds in $\Gamma$.

In view of (\ref{equ:inL}), the equation above might be rewritten as
$$\Phi\circ \Psi^{-1}(z) \; \Phi\circ \Psi^{-1}(w)= \Phi\circ \Psi^{-1}(zw), $$
where $z=\Psi(f_{1, p}\otimes_R a_{1, k})$,  $w=\Psi(g_{1, k}\otimes_R b_{1, j})$ and $zw$ denotes their product in $L_R(M)$. This proves that $\Phi\circ \Psi^{-1}$ is a ring homomorphism. Since it is an isomorphism of $R$-$R$-bimodules and preserves the gradings, it is a graded ring isomorphism.
\end{proof}

For a ring $R$, the factor category $R\mbox{-mod}/{R\mbox{-proj}}$ will be denoted by $R\mbox{-\underline{\rm mod}}$, known as the \emph{stable module category} over $R$. Similarly, for a graded ring $\Gamma$ we have the stable category  $\Gamma\mbox{-\underline{grmod}}$ of graded modules.

\begin{thm}\label{thm:s-gr}
Let $R$ be a ring and $M$ be an $R$-$R$-bimodule such that $_RM$ is finitely generated projective. Then we have an equivalence
$$\mathcal{S}(R\mbox{-{\rm mod}}, M\otimes_R-)\simeq L_R(M)\mbox{-{\rm grmod}},$$
which induces an equivalence
$$\mathcal{S}(R\mbox{-{\rm \underline{mod}}}, M\otimes_R-)\simeq L_R(M)\mbox{-}{\rm \underline{grmod}}.$$
\end{thm}

\begin{proof}
By applying Proposition~\ref{prop:stab}, the category $\mathcal{S}=\mathcal{S}(R\mbox{-{\rm mod}}, M\otimes_R-)$ has cokernels and an $\Sigma$-progenerator $\mathbf{S}(R)=(R, 0)$. By combining the equivalence in Proposition~\ref{prop:equiv-graded} and the isomorphism in Proposition~\ref{prop:iso-graded}, we infer  the first equivalence. The equivalence sends $\mathcal{S}(R\mbox{-proj}, M\otimes_R-)$ to $L_R(M)\mbox{-grproj}$, the category of finitely generated projective graded $L_R(M)$-modules. By Lemma~\ref{lem:fac}, we infer the second equivalence.
\end{proof}

\begin{rem}
Let $\mathbb{K}$ be a field. Set $R=\mathbb{K}^n$ to be a finite product of $\mathbb{K}$. Then the corresponding Leavitt ring is isomorphic to the Leavitt 
path algebra associated to the quiver with one vertex and $n$-loops; see \cite[Proposition~4.1(2)]{CWKW}.  Consequently, Theorem~\ref{thm:s-gr} describes the graded module category over the Leavitt path algebra as a stabilization of $\mathbb{K}^n\mbox{-mod}$. This description might be viewed as an enhancement of the following fact: the graded Grothendieck group of a Leavitt path algebra is isomorphic to the dimension group in the sense of Krieger; see \cite[Lemma 11]{Hazrat1}  or  \cite[Lemma 3.8]{AP}.
\end{rem}

\begin{rem}
 We assume further that $_RM$ is a finitely generated progenerator. The condition in  Lemma~\ref{lem:ss} is fulfilled. Combining it with Proposition~\ref{prop:iso-graded}, we infer that the Leavitt ring $L_R(M)$ is strongly graded. So, in view of (\ref{equ:Dade}),  we have two equivalences
  $$ \mathcal{S}(R\mbox{-{\rm mod}}, M\otimes_R-)\simeq L_R(M)_0\mbox{-{\rm mod}} \mbox{ and }  \mathcal{S}(R\mbox{-{\rm \underline{mod}}}, M\otimes_R-)\simeq L_R(M)_0\mbox{-{\rm \underline{mod}}} $$
\end{rem}

It seems that the restriction to finitely presented modules in Theorem~\ref{thm:s-gr} is essential.

\begin{rem}
    Consider the whole module category $R\mbox{-Mod}$ and the graded module category $L_R(M)\mbox{-{\rm GrMod}}$. It seems that the stabilization $\mathcal{S}(R\mbox{-{\rm Mod}}, M\otimes_R-)$ does not admit infinite coproducts. Therefore, the precise relation between $\mathcal{S}(R\mbox{-{\rm Mod}}, M\otimes_R-)$ and $L_R(M)\mbox{-{\rm GrMod}}$ is not known.
\end{rem}

\section{Noncommutative differential $1$-forms and singularity categories}

Throughout this section, we fix a ring $\Lambda$, which contains a semisimple subring $E$ such that the left $E$-module $_E\Lambda$ is finitely generated. It follows that $\Lambda$ is left artinian.

Consider the quotient $E$-$E$-bimodule $\overline{\Lambda}=\Lambda/E$. An element $a\in \Lambda$ corresponds to $\overline{a}\in \overline{\Lambda}$. The bimodule of \emph{$E$-relative noncommutative differential $1$-forms} \cite{CQ} is defined to be the following $\Lambda$-$\Lambda$-bimodule
$$\Omega_{\rm nc}=\Omega_{{\rm nc}, \Lambda/E}=\overline{\Lambda}\otimes_E \Lambda.$$
Its right $\Lambda$-action is given by $(\overline{a}\otimes_E x)b=\overline{a}\otimes_E xb$, while its left $\Lambda$-action is nontrivial and given by
$$b(\overline{a}\otimes_E x)=\overline{ba}\otimes_E x-\overline{b}\otimes_E ax.$$
By \cite[Proposition~2.5]{CQ}, $\Omega_{\rm nc}$ is projective on both sides.

For each $\Lambda$-module $X$, we have a short exact sequence.
\begin{align}\label{equ:nondif}
    0\longrightarrow \Omega_{\rm nc}\otimes_\Lambda X \stackrel{\iota_X}\longrightarrow \Lambda\otimes_E X\stackrel{\mu_X} \longrightarrow X \longrightarrow 0
\end{align}
Here, $\iota_X((\overline{a}\otimes_E b)\otimes_\Lambda x)=a\otimes_E bx-1\otimes_E abx$ and $\mu_X(a\otimes_E x)=ax$. We mention that the $\Lambda$-module $\Lambda\otimes_E X$ is projective. Therefore, we may identify $\Omega_{\rm nc}\otimes_\Lambda X $ with the first syzygy of $X$.

Following \cite[Definition~2.1]{Dam}, a ring $R$ is calld \emph{FC} if it is two-sided coherent and  satisfies
$${\rm Ext}_R^1(M, R)=0={\rm Ext}^1_{R^{\rm op}}(N, R)$$
for any finitely presented left $R$-module $M$ and finitely presented right $R$-module $N$. This terminology is justified by the fact that flat modules and coflat modules coincide for FC rings.

An additive category $\mathcal{C}$ is called \emph{Frobenius abelian} \cite{Hel, Hap}  provided that it is an abelian category with enough projective objects and enough injective objects, and that projective objects coincide with injective objects.

We mention that FC rings are coherent analogues of quasi-Frobenius rings.  The following result is well known; see \cite[Theorem~2.4]{Dam}.

\begin{lem}\label{lem:FC}
    Let $R$ be an arbitrary ring. Then the following statements are equivalent.
    \begin{enumerate}
        \item The ring $R$ is FC.
        \item The category $R\mbox{-{\rm mod}}$ is Frobenius abelian.
        \item The category $R^{\rm op}\mbox{-{\rm mod}}$ is Frobenius abelian.
    \end{enumerate}
\end{lem}

Here, we identify right $R$-modules with left modules over the opposite ring $R^{\rm op}$. Therefore, by $R^{\rm op}\mbox{-{\rm mod}}$ we mean the category of finitely presented right $R$-modules.

\begin{proof}
Recall that $R$ is left coherent if and only if $R\mbox{-mod}$ is abelian.    For ``(1) $\Rightarrow$ (2)+(3)", we use the fact that any finitely presented  $R$-modules are reflexive; see \cite[Theorem~2.4(e)]{Dam}. The implication ``(2)+(3) $\Rightarrow$ (1)" is trivial.

We will only prove ``(2) $\Rightarrow$ (3)". Assume that $R\mbox{-mod}$ is Frobenius abelian. In particular, it has kernels and an injective cogenerator $R$. By the dual of Lemma~\ref{lem:equiv} and its proof, the functor
$$ {\rm Hom}_R(-, R) \colon R\mbox{-mod} \longrightarrow (R^{\rm op}\mbox{-mod})^{\rm op} $$
is an equivalence. Then we infer (3).
\end{proof}

In what follows, we will concentrate on $\mathcal{S}=\mathcal{S}(\Lambda\mbox{-mod}, \Omega_{\rm nc}\otimes_\Lambda -)$.  Since the left $E$-module $\Lambda$ is finitely generated, the endofunctor $\Omega_{\rm nc}\otimes_\Lambda -$ on $\Lambda\mbox{-mod}$ is well defined.

\begin{prop}\label{prop:Frob}
    The category $\mathcal{S}$ is Frobenius abelian, whose subcategory formed by projective-injective objects is precisely ${\rm add}\{(\Lambda, n)\; |\; n\in \mathbb{Z}\}$.
\end{prop}

\begin{proof}
    By Lemma~\ref{lem:abel}, the category $\mathcal{S}$ is abelian. By Proposition~\ref{prop:stab}, it has enough projective objects. Moreover, the full subcategory formed by projective objects is precisely ${\rm add}\{(\Lambda, n)\; |\; n\in \mathbb{Z}\}$.

    We claim that each object $(P, n)$ is injective for $P\in \Lambda\mbox{-proj}$ and $n\in \mathbb{Z}$. It suffices to show that any monomorphism starting from $(P, n)$ splits. In view of Lemma~\ref{lem:stab} and Remark~\ref{rem:mono}, we may assume that the monomorphism is of the form
    $$[f; p]\colon (P; p)\longrightarrow (X, p),$$
    where $f\colon P\rightarrow X$ is a monomorphism and $p$ is sufficiently large. Take $g\colon X\rightarrow Y$ to be the cokernel of $f$. Consider the following commutative exact diagram of $\Lambda$-modules.
    \[\xymatrix{
0\ar[r] & P \ar[rr]^-{f} && X \ar[rr]^-{g} && Y \ar[r] & 0  \\
0\ar[r] & \Lambda\otimes_E P \ar[u]^-{\mu_P} \ar[rr]^-{\Lambda\otimes_E f} && \Lambda\otimes_E X \ar[u]^-{\mu_X} \ar[rr]^-{\Lambda\otimes_E g} && \Lambda\otimes_E Y \ar[u]_-{\mu_Y} \ar[r] & 0  \\
    0\ar[r] & \Omega_{\rm nc}\otimes_\Lambda P \ar[u]^-{\iota_P} \ar[rr]^-{\Omega_{\rm nc}\otimes_\Lambda f} && \Omega_{\rm nc}\otimes_\Lambda X  \ar[u]^-{\iota_X} \ar[rr]^-{\Omega_{\rm nc}\otimes_\Lambda g} && \Omega_{\rm nc}\otimes_\Lambda Y \ar[u]_-{\iota_Y} \ar[r] & 0
    }\]
    Since both $P$ and $\Lambda\otimes_E Y$ are projective, the monomorphisms $\iota_P$ and $\Lambda\otimes_E f$ split. It follows from the southwest commutative square that $\Omega_{\rm nc}\otimes_\Lambda f$ splits. In view of (\ref{equ:omega}), $[f; p]=[\Omega_{\rm nc}\otimes_\Lambda f; p+1]$. We deduce that the monomorphism $[f;p]$ splits, which implies the claim.

   To complete the proof,  it remains to show that each object $(X, n)$ admits a monomorphism into $(P, m)$ for some projective $\Lambda$-module $P$. In particular, it would imply that injective objects are projective. Recall from Lemma~\ref{lem:stab} that $(X, n)$ is isomorphic to $(\Omega_{\rm nc}\otimes_\Lambda X, n+1)$. The latter object  embeds into $(\Lambda\otimes_E X, n+1)$ by (\ref{equ:nondif}). This completes the proof.
\end{proof}

We observe that the abelian category $\mathcal{S}$ is not noetherian in general.

\begin{exm}\label{exm:noe}
    Let $\mathbb{K}$ be a field  and $\Lambda$ be a finite dimensional $\mathbb{K}$-algebra. Assume that ${\rm dim}_\mathbb{K}\; \Lambda=d\geq 3$. In $\mathcal{S}$, we have isomorphisms
    $$(\Lambda, 0)\simeq (\Omega_{\rm nc}, 1)\simeq (\Lambda, 1)^{\oplus d-1},$$
    where the isomorphism on the right side follows from an isomorphism $\Omega_{\rm nc}\simeq \Lambda^{\oplus d-1}$ of left $\Lambda$-modules.   Using the composite isomorphism and by induction, we construct an infinite strictly ascending chain of subobjects of $(\Lambda, 0)$. It follows that $\mathcal{S}$ is non-noetherian.
\end{exm}

Recall that the looped category $(\Lambda\mbox{-\underline{\rm mod}}, \Omega_{\rm nc}\otimes_\Lambda-)$ is \emph{left triangulated} \cite{KV, BM}, whose left triangles are induced by short exact sequences of $\Lambda$-modules. To be more precise, each exact sequence $0\rightarrow X\stackrel{f} \rightarrow Y \stackrel{g} \rightarrow Z \rightarrow 0$ gives rise to the following commutative diagram.
\[\xymatrix{
0\ar[r] & \Omega_{\rm nc}\otimes_\Lambda Z \ar@{.>}[d]_-{\omega}\ar[rr]^-{\iota_Z} && \Lambda\otimes_E Z \ar@{.>}[d]\ar[rr]^-{\mu_Z} && Z\ar@{=}[d] \ar[r] & 0\\
0\ar[r] & X \ar[rr]^-{f} && Y \ar[rr]^-{g} && Z \ar[r] & 0
}\]
Then we have the induced left triangle
\begin{align}\label{equ:left-tri}
    \Omega_{\rm nc}\otimes_\Lambda Z \stackrel{\underline{\omega}} \longrightarrow X \stackrel{\underline{f}} \longrightarrow Y \stackrel{\underline{g}} \longrightarrow Z
\end{align}
in $\Lambda\mbox{-\underline{\rm mod}}$. Up to isomorphism, all left triangles arise in this manner.

The left triangulated structure above induces a triangulated structure on the  stabilization $\mathcal{S}(\Lambda\mbox{-\underline{\rm mod}}, \Omega_{\rm nc}\otimes_\Lambda-)$.  Its suspension functor is given by $\Sigma$, which sends $(X, n)$ to $(X, n+1)$. For each integer $n$, the left triangle (\ref{equ:left-tri}) gives rise to a triangle
\begin{align}\label{equ:tri}
    (X, n) \xrightarrow{[\underline{f}; n]} (Y, n) \xrightarrow{[\underline{g}; n]} (Z, n) \xrightarrow{-[\underline{\omega}; n+1]} (X, n+1)
\end{align}
 in $\mathcal{S}(\Lambda\mbox{-\underline{\rm mod}}, \Omega_{\rm nc}\otimes_\Lambda-)$. For details, we refer to \cite[Section~3]{Bel}.

 Denote by $\mathbf{D}^b(\Lambda\mbox{-mod})$ the bounded derived category of $\Lambda\mbox{-mod}$, whose suspension functor is also denoted by $\Sigma$. We identify a $\Lambda$-module $M$ with the stalk complex concentrated in degree zero. Therefore, $\Sigma^n(M)$ is a stalk complex concentrated in degree $-n$.

 The bounded homotopy category $\mathbf{K}^b(\Lambda\mbox{-proj})$ of finitely generated projective modules is viewed as a triangulated subcategory of $\mathbf{D}^b(\Lambda\mbox{-mod})$. Following \cite[Definition~1.2.2]{Buc} and \cite[Definition~1.8]{Orl}, the \emph{singularity category}  of $\Lambda$ is defined to the following Verdier quotient triangulated category
 $$\mathbf{D}_{\rm sg}(\Lambda)=\mathbf{D}^b(\Lambda\mbox{-mod})/{\mathbf{K}^b(\Lambda\mbox{-proj})}.$$

 Denote by $Q\colon \mathbf{D}^b(\Lambda\mbox{-mod})\rightarrow \mathbf{D}_{\rm sg}(\Lambda)$ the quotient functor. The composite functor
 $$\Lambda\mbox{-mod} \hookrightarrow \mathbf{D}^b(\Lambda\mbox{-mod})\stackrel{Q} \longrightarrow  \mathbf{D}_{\rm sg}(\Lambda)$$
 vanishes on projective modules. So, we have a well-defined functor
 $$F\colon \Lambda\mbox{-}\underline{\rm mod} \longrightarrow  \mathbf{D}_{\rm sg}(\Lambda).$$

The following fundamental result is due to \cite[Theorem~6.5.3]{Buc} in  Gorenstein cases and \cite[Section~2]{KV} in general, whose detailed proof is found in \cite[Section~4]{Bel}.

\begin{thm}\label{thm:KV}{\rm (Buchweitz, Keller-Vossieck)}
 The functor above $F$ induces a triangle equivalence
 $$\mathcal{S}(\Lambda\mbox{-\underline{\rm mod}}, \Omega_{\rm nc}\otimes_\Lambda-) \stackrel{\sim} \longrightarrow  \mathbf{D}_{\rm sg}(\Lambda), $$
 which sends $[\underline{f}; n]\colon (X, n) \rightarrow (Y, n)$ to $Q\Sigma^n(f)\colon Q\Sigma^n(X)\rightarrow Q\Sigma^n(Y)$.
\end{thm}

Here, we identify the induced tensor endofunctor $\Omega_{\rm nc}\otimes_\Lambda-$ on $\Lambda\mbox{-\underline{\rm mod}}$ with the syzygy endofunctor.

Recall that $\mathcal{S}=\mathcal{S}(\Lambda\mbox{-mod}, \Omega_{\rm nc}\otimes_\Lambda -)$. Denote by $\underline{\mathcal{S}}$ its stable category, that is, the factor category of $\mathcal{S}$ modulo projective objects. Since $\mathcal{S}$ is Frobenius, the stable category $\underline{\mathcal{S}}$ is canonically triangulated \cite[Chapter~I, Section~2]{Hap}.

\begin{prop}\label{prop:tri-equ}
    There is a triangle equivalence $\underline{\mathcal{S}}\simeq \mathcal{S}(\Lambda\mbox{-\underline{\rm mod}}, \Omega_{\rm nc}\otimes_\Lambda-) $.
\end{prop}

\begin{proof}
    Consider the canonical projection $\pi\colon \mathcal{S} \rightarrow \mathcal{S}(\Lambda\mbox{-\underline{\rm mod}}, \Omega_{\rm nc}\otimes_\Lambda-)$. By Lemma~\ref{lem:abel} and Remark~\ref{rem:mono}, each short exact sequence in $\mathcal{S}$ is isomorphic to the following one
    $$0\longrightarrow (X, n) \stackrel{[f;n]} \longrightarrow (Y, n) \stackrel{[g; n]} \longrightarrow (Z, n) \longrightarrow 0,$$
    where $0\rightarrow X \stackrel{f}\rightarrow Y\stackrel{g} \rightarrow Z \rightarrow 0$ is an exact sequence of $\Lambda$-modules. In view of the triangle (\ref{equ:tri}), we conclude that $\pi$ sends short exact sequences to triangles. More precisely, the functor $\pi$ is a $\partial$-functor in the sense of \cite[Section~1]{Kel}. By \cite[Lemma~2.5]{Chen11} the induced functor $\underline{\mathcal{S}}\rightarrow \mathcal{S}(\Lambda\mbox{-\underline{\rm mod}}, \Omega_{\rm nc}\otimes_\Lambda-)$ is a triangle functor. It is an equivalence by Lemma~\ref{lem:fac}.
\end{proof}

\begin{rem}
    Combining Theorem~\ref{thm:KV} with Proposition~\ref{prop:tri-equ}, we infer a triangle equivalence $\underline{\mathcal{S}}\simeq \mathbf{D}_{\rm sg}(\Lambda)$. In other words, we obtain an explicit \emph{Frobenius enhancement} $\mathcal{S}$ for $\mathbf{D}_{\rm sg}(\Lambda)$. By the general theory on dg quotients  \cite{Kel99, Dri}, it is well known that $\mathbf{D}_{\rm sg}(\Lambda)$ is algebraic  and thus admits Frobenius enhancements; see also \cite{CW}. However, it is highly nontrivial whether  such Frobenius enhancements are unique in some reasonable sense; compare \cite{CS}.
\end{rem}

The following main result  justifies the title.

\begin{thm}\label{thm:main}
    Keep the assumptions above. Recall $\Omega_{\rm nc}=\Omega_{{\rm nc},\Lambda/E}$. Then the Leavitt ring $L_\Lambda(\Omega_{\rm nc})$ is strongly graded, whose zeroth component $L_\Lambda(\Omega_{\rm nc})_0$ is an FC ring. Moreover, we have a triangle equivalence
    $$\mathbf{D}_{\rm sg}(\Lambda)\simeq L_\Lambda(\Omega_{\rm nc})_0\mbox{-\underline{\rm mod}}.$$
\end{thm}

\begin{proof}
    Since $\Lambda$ is left artinian, there are only finitely many indecomposable projective $\Lambda$-modules. It follows that the descending chain in $\Lambda\mbox{-mod}$
    $${\rm add}(\Lambda)\supseteq {\rm add}(\Omega_{\rm nc}\otimes_\Lambda \Lambda) \supseteq {\rm add}(\Omega_{\rm nc}^{\otimes_\Lambda 2}\otimes_\Lambda \Lambda)\supseteq \cdots$$
   necessarily stabilizes. It follows from Lemma~\ref{lem:ss} that the orbit ring $\Gamma(\mathbf{S}(\Lambda); \Sigma)$ is strongly graded. By Proposition~\ref{prop:iso-graded}, we infer that $L=L_\Lambda(\Omega_{\rm nc})$ is also strongly graded. By combining the equivalences in Theorem~\ref{thm:s-gr} and (\ref{equ:Dade}), we deduce an equivalence
    \begin{align}\label{equ:S-L}
    \mathcal{S}\simeq L_0\mbox{-mod},
    \end{align}
    where $L_0=L_\Lambda(\Omega_{\rm nc})_0$.
    By Proposition~\ref{prop:Frob}, we infer that $L_0\mbox{-mod}$ is a Frobenius abelian category. It follows from Lemma~\ref{lem:FC} that $L_0$ is an FC ring.

    The equivalence (\ref{equ:S-L}) induces a triangle equivalence $\underline{\mathcal{S}}\simeq L_0\mbox{-\underline{mod}}$. By combining Theorem~\ref{thm:KV} and Proposition~\ref{prop:tri-equ}, we deduce the required triangle equivalence.  \end{proof}

\begin{rem}\label{rem:final}
 By the equivalence (\ref{equ:S-L}) and Example~\ref{exm:noe}, the ring $L_0$ is not noetherian in general. It follows that the FC ring $L_0$ is not quasi-Frobenius in general.
\end{rem}

Recall that a ring $R$ is called \emph{von Neumann regular} if any finitely presented $R$-module is projective. Any von Neumann regular ring is FC.  We have the following immediate consequence of Theorem~\ref{thm:main}.

\begin{cor}\label{cor:regular}
Keep the assumptions above. Then the left artinian ring $\Lambda$ has finite global dimension if and only if the  ring $L_\Lambda(\Omega_{\rm nc})_0$ is von Neumann regular.
\end{cor}

\begin{proof}
    Recall that $\Lambda$ has finite global dimension if and only if the singularity category $\mathbf{D}_{\rm sg}(\Lambda)$ vanishes. The stable module category $L_\Lambda(\Omega_{\rm nc})_0\mbox{-\underline{\rm mod}}$ vanishes if and only if $L_\Lambda(\Omega_{\rm nc})_0$ is von Neumann regular. Then we are done by the triangle equivalence in Theorem~\ref{thm:main}.
\end{proof}

 The corollary above indicates that the non-regularity of $L_\Lambda(\Omega_{\rm nc})_0$ detects the homological singularity of $\Lambda$.

 \section{An example}
 \label{section:anexample}

In this section, we describe the Leavitt ring $L_\Lambda(\Omega_{\rm nc})$ of an algebra $\Lambda$ with radical square zero. It turns out that its zeroth component $L_0$ is a trivial extension of a von Neumann regular ring.

 Recall that a quiver $Q=(Q_0, Q_1; s,t)$ consists of the following data: a set $Q_0$ of vertices, a set $Q_1$ of arrows, and the maps $s, t\colon Q_1\rightarrow Q_0$ assign to each arrow $\alpha$ its starting vertex $s(\alpha)$ and terminating vertex $t(\alpha)$. Let $\mathbb{K}$ be a field. Denote by $\mathbb{K}Q$ its path algebra, which has a $\mathbb{K}$-basis given by paths in $Q$. Each vertex $i$ corresponds to a trivial path $e_i$, which becomes an idempotent in $\mathbb{K}Q$. We will assume that $Q$ is finite, that is, both sets $Q_0$ and $Q_1$ are finite. In this situation, the path algebra $\mathbb{K}Q$ is unital with $1=\sum_{i\in Q_0} e_i$.

For a finite quiver $Q$, denote by $\Lambda=\mathbb{K}Q/{\mathbb{K}Q_{\geq 2}}$ the corresponding algebra with radical square zero. Here, $\mathbb{K}Q_{\geq 2}$ denotes the two-sided ideal generated by paths of length two. The following relations
\begin{align}\label{rel:0}
\beta\alpha=0 \mbox{ for any } \alpha, \beta\in Q_1 \mbox{ satisfying } t(\alpha)=s(\beta)
\end{align}
hold in $\Lambda$. Here, we write the concatenation of paths from righ to left.

Set $E=\mathbb{K}Q_0$, which is viewed as a semisimple subalgebra of $\Lambda$. As an $E$-$E$-bimodule, we have $\Lambda=E\oplus \mathbb{K}Q_1$. We identify $\overline{\Lambda}=\Lambda/E$ with $\mathbb{K}Q_1$. Therefore, the bimodule $\Omega_{\rm nc}=\overline{\Lambda}\otimes_E \Lambda$ of  $E$-relative noncommutative differential $1$-forms is identified with $\mathbb{K}Q_1\otimes_E \Lambda$. It has a $\mathbb{K}$-basis given by
$$\{ \alpha \otimes_E e_{s(\alpha)}, \; \beta \otimes_E \alpha\; |\; \alpha, \beta \in Q_1,\  t(\alpha)=s(\beta) \}.$$
Set $\bar{\alpha}=\alpha \otimes_E e_{s(\alpha)}$. The set $\{\bar{\alpha}\; |\; \alpha\in Q_1\}$ generates the bimodule $\Omega_{\rm nc}$, which satisfies the following relations:
\begin{align}\label{rel:1}
    \beta \bar{\alpha}+\bar{\beta}\alpha=0 \mbox{ for any } \alpha, \beta\in Q_1 \mbox{ satisfying } t(\alpha)=s(\beta).
\end{align}

Consider the left-dual bimodule $\Omega_{\rm nc}^*={\rm Hom}_\Lambda (\Omega_{\rm nc}, \Lambda)$. Each $\alpha\in Q_1$ determines a unique element $\bar{\alpha}^*\in \Omega_{\rm nc}^*$ satisfying
$$\bar{\alpha}^*(\bar{\alpha'})=\delta_{\alpha, \alpha'} \; e_{t(\alpha)}.$$
We observe that $\bar{\alpha}^*(\beta\otimes_E \alpha)=-\beta$. In the bimodule $\Omega_{\rm nc}^*$, we have $\bar{\alpha}^*= e_{s(\alpha)} \bar{\alpha}^* e_{t(\alpha)}$. The bimodule $\Omega_{\rm nc}^*$ has a $\mathbb{K}$-basis given by
$$\{\bar{\alpha}^*, \bar{\alpha}^*\alpha'\; |\; \alpha, \alpha'\in Q_1 \mbox{ satisfying } t(\alpha)=t(\alpha')\}.$$
In particular,  the set $\{\bar{\alpha}^*\; |\; \alpha\in Q_1\}$ generates $\Omega_{\rm nc}^*$ as a right $\Lambda$-module. These elements satisfy the following relations.
\begin{align}\label{rel:2}
  \alpha' \bar{\alpha}^* =\delta_{\alpha, \alpha'} \sum_{\{\beta\in Q_1\; |\; s(\beta)=t(\alpha)\}} \bar{\beta}^*\beta.
\end{align}
Here, if $t(\alpha)$ is a sink, that is, the set $\{\beta\in Q_1\; |\; s(\beta)=t(\alpha)\}$ is empty, the relation above is understood as $\alpha \bar{\alpha}^*=0$.

The set $\{\bar{\alpha}, \bar{\alpha}^*\; |\; \alpha\in Q_1\}$ forms a dual basis for $\Omega_{\rm nc}$. Therefore, we have the Casimir element
$$c=\sum_{\alpha \in Q_1} \bar{\alpha}^*\otimes_E \bar{\alpha}. $$

Denote by $\widetilde{Q}$ the quiver obtained from $Q$ by adding for each $\alpha\in Q_1$, a new parallel arrow $\bar{\alpha}$ and the reverse arrow $\bar{\alpha}^*$. We will consider the \emph{Cuntz-Krieger relations} in $\mathbb{K}\widetilde{Q}$.
\begin{align}\label{rel:CK}
    \bar{\alpha'}\bar{\alpha}^*=\delta_{\alpha, \alpha'}\; e_{t(\alpha)}, \mbox{ and }  \sum_{\alpha\in Q_1} \bar{\alpha}^*\bar{\alpha}=1.
\end{align}

Denote by $I$ the two-sided ideal of $\mathbb{K}\widetilde{Q}$ generated by the relations (\ref{rel:0}), (\ref{rel:1}), (\ref{rel:2}) and the Cuntz-Krieger relations (\ref{rel:CK}). The path algebra $\mathbb{K}\widetilde{Q}$ is naturally $\mathbb{Z}$-graded by means of ${\rm deg}(\alpha)=0$, ${\rm deg}(\bar{\alpha})=-1$ and ${\rm deg}(\bar{\alpha}^*)=1$. The ideal $I$ is homogeneous,  and thus the quotient algebra $\mathbb{K}\widetilde{Q}/I$ is $\mathbb{Z}$-graded.

 \begin{prop}\label{prop:radzero}
     Recall that $\Lambda=\mathbb{K}Q/{\mathbb{K}Q_{\geq 2}}$ and $E=\mathbb{K}Q_0$. Then as a graded algebra,  the Leavitt ring $L_\Lambda(\Omega_{\rm nc})$ is isomorphic to the quotient algebra $\mathbb{K}\widetilde{Q}/I$.
 \end{prop}

 \begin{proof}
     The tensor ring $T_\Lambda(\Omega_{\rm nc}^*\oplus \Omega_{\rm nc})$ is isomorphic to the quotient algebra of $\mathbb{K}\widetilde{Q}$ modulo the relations (\ref{rel:0}), (\ref{rel:1}) and  (\ref{rel:2}). The two generators in the defining  ideal of the Leavitt ring correspond to the two  Cuntz-Krieger relations above. We omit the details. 
 \end{proof}

\begin{rem}
    Set $\bar{Q}_1=\{\bar{\alpha} \; |\; \alpha\in Q_1\}$. Then $\mathbb{K}\bar{Q}_1$ is naturally an $E$-$E$-bimodule. Consider the graded subalgebra  $U$ of $\mathbb{K}\widetilde{Q}/I$ generated by $\{e_i, \bar{\alpha}, \bar{\alpha}^*\; |\; i\in Q_0, \alpha\in Q_1\}$. We have a decomposition
    $$ \mathbb{K}\widetilde{Q}/I=U\oplus (\Sigma_{\alpha\in Q_1}U\alpha)$$
    of $U$-$U$-bimodules. The complement $V=\Sigma_{\alpha\in Q_1}U\alpha$ of $U$ in  $\mathbb{K}\widetilde{Q}/I$ is square zero. In other words, $\mathbb{K}\widetilde{Q}/I$ is a trivial extension $U\ltimes V$ of $U$ by the $U$-$U$-bimodule $V$.

    We observe that $U$ is isomorphic to the Leavitt ring $L_E(\mathbb{K}\bar{Q}_1)$ associated to the $E$-$E$-bimodule $\mathbb{K}\bar{Q}_1$, and thus  isomorphic to a certain Leavitt path algebra; see \cite[Proposition~4.1(2)]{CWKW}. Consequently, its zeroth component $U_0$ is von Neumann regular; see \cite[Lemma~4.1(2)]{CY}.

    By the isomorphism in Proposition~\ref{prop:radzero}, we infer that the zeroth component $L_\Lambda(\Omega_{\rm nc})_0$ is isomorphic to a trivial extension $U_0\ltimes  V_0$ of a von Neumann regular ring $U_0$ by a certain $U_0$-$U_0$-bimodule $V_0$. However, the bimodule structure of $V_0$ is not well understood.

    In view of Theorem~\ref{thm:main}, we actually obtain an example of a certain trivial extension of a von Neumann regular ring being an FC ring. This might be analogous to the well-known fact that the trivial extension of any artinian  algebra by the dual of the regular bimodule is always symmetric; see \cite[Chapter~II, Proposition~3.9]{ARS}.
\end{rem}

 \vskip 15pt
\noindent{\bf Acknowledgements}. \quad  We thank Pere Ara, Roozbeh Hazrat, Martin Kalck and Huanhuan Li for helpful comments. We also thank Martin Kalck for suggesting the example in Section~\ref{section:anexample}. This project is supported by National Key R$\&$D Program of China (No.s 2024YFA1013801 and  2024YFA1013803), the Fundamental Research Funds for the Central Universities (No. 020314380037) and the National Natural Science Foundation of China (No.s 12325101, 12131015, 13004005 and 12371043).

\bibliography{}

\vskip 10pt

 {\footnotesize \noindent  Xiao-Wu Chen\\
 School of Mathematical Sciences, University of Science and Technology of China, Hefei 230026, Anhui, PR China\\

 \footnotesize \noindent Zhengfang Wang\\
 School of Mathematics, Nanjing University, Nanjing 210093, Jiangsu, PR China}

\end{document}